\documentclass[preprint,12pt]{elsarticle}
\usepackage{amsfonts}
\usepackage{amsmath}
\usepackage{amssymb}

\newdefinition{definition}{Definition}
\newdefinition{remark}{Remark}
\newtheorem{theorem}{Theorem}
\newtheorem{corollary}{Corollary}
\newtheorem{proposition}{Proposition}
\newproof{proof}{Proof}

\usepackage{amssymb}

\begin{document}

\begin{frontmatter}



\title{A Note on the Isotopism of Commutative Semifields}


\author{Yue Zhou}
\ead{yue.zhou@st.ovgu.de}

\address{Department of Mathematics, Otto-von-Guericke-University
Magdeburg, \\39106 Magdeburg, Germany}

\begin{abstract}
We present an example of two isotopic but not strongly isotopic commutative semifields. This example shows that a recent result of Coulter and Henderson on semifield of order $p^n$, $n$ odd, can not be generalized to the case $n$ even.
\end{abstract}


\end{frontmatter}



\section{Introduction}\label{sec_intro}

A \emph{semifield} $F$ is an algebraic structure satisfying all the axioms for a skewfield except (possibly) associativity.
A finite field is a trivial example of a semifield. Furthermore, if $F$ does not necessarily have a multiplicative identity, then it is called a \emph{presemifield}. A semifield is not necessarily commutative or associative. However, by Wedderburn's Theorem \cite{Wedderburn1905}, in the finite case, associativity implies commutativity. Therefore, a non-associative finite commutative semifield is the closest algebraic structure to a finite field.

In the earlier literature, semifields were also called \emph{division rings} or \emph{distributive quasifields}. The study of semifields was initiated by Dickson \cite{Dickson1906}, shortly after the classification of the finite fields. Until now, semifields have become an attracting topic in many different areas of mathematics, such as difference sets, coding theory and finite geometry.

The first non-trivial semifields were constructed by Dickson \cite{Dickson1906}. In \cite{Knuth1963}, Knuth showed that the additive group of a semifield $F$ is an elementary abelian group, and the additive order of the elements in $F$ is called the characteristic of $F$. Hence, any finite semifield can be represented by $(\mathbb{F}_{p^n}, +, *)$. Here $(\mathbb{F}_{p^n}, +)$ is the additive group of the finite field $\mathbb{F}_{p^n}$ and $x*y=\varphi(x,y)$, where $\varphi$ is a mapping from $\mathbb{F}_{p^n}\times \mathbb{F}_{p^n}$ to $\mathbb{F}_{p^n}$.

On the other hand, there is a well-known correspondence, via coordinatisation, between commutative semifields and translation planes of Lenz-Barlotti type V.1, see \cite{HughesPiper}. In \cite{Albert1960}, Albert showed that two semifields coordinatise isomorphic planes if and only if they are isotopic:
\begin{definition}
    Let $F_1=(\mathbb{F}_{p^n}, +, *)$ and $F_2=(\mathbb{F}_{p^n}, +, \star)$ be two presemifields. If there exist three linearized permutation polynomials $L, M, N\in\mathbb{F}_{p^n}[x]$ such that
    $$M(x)\star N(y)=L(x*y)$$
    for any $x,y\in\mathbb{F}_{p^n}$, then $F_1$ and $F_2$ are called \emph{isotopic}, and the triple $(M,N,L)$ is an \emph{isotopism } between $F_1$ and $F_2$. Furthermore, if there exists an isotopism of the form $(N, N, L)$ between $F_1$ and $F_2$, then $F_1$ and $F_2$ are \emph{strongly isotopic}.
\end{definition}

We refer the reader to \cite{LidlNiederreiter} for more background on finite fields, in particular about linearized polynomials. Let $F=(\mathbb{F}_{p^n}, +, *)$ be a presemifield, and $a\in F$. If we define a new multiplication $\star$ by the rule
    $$(x*a)\star(a*y)=x*y,$$
we obtain a semifield $(\mathbb{F}_{p^n}, +, \star)$ with unit $a*a$. There are many semifields associated with a presemifield, but they are all isotopic.

Let $F=(\mathbb{F}_{p^n},+,*)$ be a semifield. The subsets
\begin{equation*}
\begin{aligned}
  N_l(F)=\{a\in F: (a*x)*y=a*(x*y) \text{ for all }x,y\in F\},\\
  N_m(F)=\{a\in F: (x*a)*y=x*(a*y) \text{ for all }x,y\in F\},\\
  N_r(F)=\{a\in F: (x*y)*a=x*(y*a) \text{ for all }x,y\in F\},
\end{aligned}
\end{equation*}
are called the \emph{left, middle} and \emph{right nucleus} of $F$, respectively. It is easy to check that these sets are finite fields. The subset $N(F)=N_l(F)\cap N_m(F) \cap N_r(F)$ is called the \emph{nucleus} of $F$. It is easy to see, if $F$ is commutative, then $N_l(F)=N_r(F)=N(F)$. In \cite{HughesPiper}, the geometry interpretations of these nuclei are presented.

Next, we give the definition of planar functions, which was introduced by Dembowski and Ostrom in \cite{DO1968} to describe affine planes possessing a collineation group with specific properties.
\begin{definition}
    Let $p$ be an odd prime. A function $f: \mathbb{F}_{p^n}\rightarrow\mathbb{F}_{p^n}$ is called a \emph{planar function}, or \emph{perfect nonlinear (PN)}, if for each $a\in \mathbb{F}_{p^n}^*$, $f(x+a)-f(x)$ is a bijection on $\mathbb{F}_{p^n}$.
\end{definition}

For $p=2$, if $x_0$ is a root of $f(x+a)-f(x)=b$, then $x_0+a$ is another one, hence there is no planar functions over $\mathbb{F}_{2^n}$.
A \emph{Dembowski-Ostrom} (DO) polynomial $D\in \mathbb{F}_{p^n}[x] $ is a polynomial
$$D(x)=\sum_{i,j}a_{ij}x^{p^i+p^j}\enspace .$$
Obviously, $D(x+a)-D(x)-D(a)$ is a linearized polynomial for any nonzero $a$. It can be proved that a planar DO polynomial is equivalent to a commutative presemifield with odd characteristic, see \cite{Coulter2008}. In fact, if $*$ is the presemifield product, then the corresponding planar function is $f(x)=x*x$; when the planar DO polynomial $f$ is given, then the corresponding presemifield product can be defined as
\begin{equation}\label{eq_PN->production}
    x*y=\frac{1}{2}(f(x+y)-f(x)-f(y))\enspace .
\end{equation}

A function from a finite field $\mathbb{F}_{p^n}$ to itself is \emph{affine}, if it is defined by the sum of a constant and a linearized polynomial over $\mathbb{F}_{p^n}$. There are several equivalence relations of functions for which the \emph{planar} property is invariant:
\begin{definition}
    Two functions $f$ and $g: \mathbb{F}_{p^n}\rightarrow \mathbb{F}_{p^n}$ are called
    \begin{itemize}
      \item \emph{extended affine equivalent} (EA-equivalent), if $g = l_1 \circ f \circ l_2+l_3$, where $l_1, l_2$ and $l_3$ are affine functions, and where $l_1,l_2$ are permutations of $\mathbb{F}_{p^n}$. Furthermore, if $l_3$ is the zero mapping, then $f$ and $g$ are called \emph{affine equivalent};
      \item \emph{Carlet-Charpin-Zinoviev equivalent} (CCZ-equivalent or graph equivalent), if there is some affine permutation $L$ of $\mathbb{F}_{p}^{2n}$, such that $L(G_f)=G_g$, where $G_f=\{(x,f(x):x\in\mathbb{F}_{p^n})\}$ and $G_g=\{(x,g(x):x\in\mathbb{F}_{p^n})\}$.
    \end{itemize}
\end{definition}

Generally speaking, EA-equivalence implies CCZ-equivalence, but not vice versa, see \cite{Budaghyan06}. However, if planar functions $f$ and $g$ are CCZ-equivalent, then they are also EA-equivalent \cite{B-H2008,EA=CCZ_PN2008}.  Because of the correspondence between commutative presemifields with odd characteristic and planar functions as we mentioned above, the strong isotopism of two commutative presemifields is equivalent to the affine equivalence of the corresponding  planar DO functions, which we call directly the \emph{equivalence} of planar DO functions.

\section{Isotopism $\neq$ Strong Isotopism}
In \cite{Coulter2008}, Coulter and Henderson proved the following theorem.
\begin{theorem}\label{th_CoulterHenderson}
    Let $F_1 =(\mathbb{F}_q, +, \star)$ and $F_2 =(\mathbb{F}_q, +, *)$ be isotopic commutative semifields. Then there exists an isotopism $(M, N, L)$ between $F_1$ and $F_2$ such that either
    \begin{enumerate}
      \item $M=N$, or
      \item $M(x)\equiv \alpha\star N(x) \mod (x^q-x)$, where $\alpha \in N_m(F_1)\setminus N(F_1)$ cannot be written in the form $\alpha=\gamma\star\beta^2$ where $\gamma\in N(F_1)$ and $\beta\in N_m(F_1)$.
    \end{enumerate}
\end{theorem}

It implies that any commutative semifield can generate at most two non-strongly isotopic commutative semifields. Some important corollaries are also presented in \cite{Coulter2008}, for example,
\begin{corollary}
    Any two commutative semifields of order $p^e$ with $e$ odd are isotopic if and only if they are strongly isotopic.
\end{corollary}

Pieper-Seier and Spille  \cite{pieper-seier_remarks_1999} showed that the Cohen-Ganley commutative semifield \cite{Cohen-Ganley1982} has exactly two classes of strong isotopy. In this paper, we present another example\footnote{In the previous version of this paper, we have claimed that our example is the first one. However, later Coulter and Knarr informed us about the result from \cite{pieper-seier_remarks_1999}}.

First, we introduce a family of planar functions:
$$\frac{1}{2}(\mathrm{Tr}(x^2)+G(x^{q^2+1}))$$
over $\mathbb{F}_{q^{2m}}$, where $q$ is a power of an odd prime $p$, $m=2k+1$, $\mathrm{Tr}(\cdot)$ is the trace function from $\mathbb{F}_{q^{2m}}$ to $\mathbb{F}_{q^m}$, and $G(x)=h(x-x^{q^m})$, where $h\in \mathbb{F}_{q^{2m}}[x]$ is defined as
$$h(x) = \sum_{i=0}^{k}(-1)^i x^{q^{2i}}+\sum_{j=0}^{k-1}(-1)^{k+j}x^{q^{2j+1}}\enspace.$$
This planar function family corresponds to Bierbrauer's generalization of the semifield discovered by Lunardon, Marino, Polverino and Trombetti over $q^{6}$, see \cite{Bierbrauer2009-3,LMPT2009}. Hence the corresponding semifield  should be called Lunardon-Marino-Polverino-Trombetti-Bierbrauer (LMPTB) semifields \cite{Bierbrauer2009-3}.

Let $q=3$, $m=3$ and $\mathbb{F}_{3^6}=\mathbb{F}_{3}(\xi)$, where $\xi$ is a root of $x^6-x^4+x^2-x-1\in \mathbb{F}_{3}[x]$. Let $F_1=(\mathbb{F}_{3^6}, +, \star)$ be the LMPTB semifield. By MAGMA\cite{Magma}, we calculate that $| N_m(F_1)| =3^2$ and $| N(F_1)| =3$, and there are four $\alpha \in N_m(F_1)\setminus N(F_1)$, which cannot be written in the form $\alpha=\gamma\star\beta^2$, where $\gamma\in N(F_1)$ and $\beta\in N_m(F_1)$. They are $\lambda, \lambda^3, \lambda^5$ and $\lambda^7$, where $\lambda=\xi^{91}$.

Now, we can define another semifield $F_2$ with the multiplication given by
$$x\odot y=(\lambda\star x)\star y$$
Obviously, $F_1$ and $F_2$ are isotopic. As we mentioned above, to tell whether $F_1$ and $F_2$ are strongly isotopic, we just need to calculate whether $f_1(x)=x\star x$ is equivalent to $f_2(x)=x\odot x$. By Lagrange interpolation, we have
\begin{equation*}
    f_1(x)=x^{270} - x^{246} + x^{90} - x^{82} - x^{54} + x^{30} - x^{10} -
    x^{2},
\end{equation*}
\begin{equation*}
    f_2(x)= \lambda^3(x^{270} - x^{246} - \lambda^2x^{90} + \lambda^2x^{82} -
       x^{54} + x^{30} + \lambda^2x^{10} + \lambda^2x^2)\enspace .
\end{equation*}

Let $f:\mathbb{F}_{p^n}\rightarrow \mathbb{F}_{p^n}$ be any function. Since the additive group of $\mathbb{F}_{p^n}$ is the linear space $\mathbb{F}_p^n$, $f$ can also be considered as a mapping from $\mathbb{F}_p^n$ to itself. Define a matrix $M_f\in \mathbb{F}_p^{(2n+1, p^n)}$ as follows:
\begin{equation}\label{eq_Mf}
  M_f=\left(
               \begin{array}{ccc}
                 \cdots & 1 & \cdots \\
                 \cdots & x & \cdots \\
                 \cdots & f(x) & \cdots \\
               \end{array}
             \right)_{x\in\mathbb{F}_p^n}
\end{equation}
Then we can construct a linear code $C_f$ over $\mathbb{F}_p$ by the generator matrix $M_f$. Furthermore, it can be proved that
\begin{proposition}
     Let $p$ be a prime, and $n$ be an integer. Two functions $f,g:\mathbb{F}_{p^n}\rightarrow \mathbb{F}_{p^n}$ are CCZ-equivalent, if and only if the corresponding codes $C_f$ and $C_g$ are permutation equivalent.
\end{proposition}
\begin{proof}
    Assume that $C_f$ and $C_g$ are permutation equivalent, then we have a permutation matrix $P$ and a $(2n+1)\times(2n+1)$ matrix $L$ with full rank, such that
    $$L\cdot M_f\cdot P=M_g\enspace.$$
    That means there are $u,v\in\mathbb{F}_p^n$ and a matrix $\tilde{L}$ with full rank, such that
    $$\tilde{L}\cdot \left(
               \begin{array}{ccc}
                 \cdots & x & \cdots \\
                 \cdots & f(x) & \cdots \\
               \end{array}
             \right)\cdot P=\left(
               \begin{array}{ccc}
                 \cdots & x & \cdots \\
                 \cdots & g(x) & \cdots \\
               \end{array}
             \right)+\left(
               \begin{array}{c}
                 u \\
                 v \\
               \end{array}
             \right) \enspace.
    $$
    Therefore, by the definition of CCZ-equivalence, $f$ and $g$ are CCZ-equivalent. The proof of the converse is the same. \qed

\end{proof}

For the equivalence of codes, including permutation equivalence
and monomial equivalence, see \cite{HuffmanPless2003}.

It is well-known that function $f$ mapping $\mathbb{F}_{p^n}$ to itself
is planar if and only if for every nonzero $a\in\mathbb{F}_{p^n}$,
the function $\mathrm{Tr}(af(x))$ is generalized bent, see \cite{CarletDubuc1999}. For planar DO-polynomials, it is equivalent to the nonsingularity of $\mathrm{Tr}(af(x))$ as a p-ary quadratic form, for every nonzero $a\in\mathbb{F}_{p^n}$. Therefore, the weight distribution of $C_f$ can be deduced, see \cite{LiQuLing2009}, and there are only $p-1$ code words $(i,i,\cdots,i)$ with weight $p^n$ ($0<i<p$). Thus, for the codes $C_f$ and $C_g$ from the planar functions $f$ and $g$, monomial and permutation equivalence are identical. By MAGMA, we calculated that $C_{f_1}$ is not monomially equivalent to $C_{f_2}$ (MAGMA only offers the command to tell the monomial equivalence of two linear codes, that is why we emphasize the identity of monomial and permutation equivalence between $C_f$ and $C_g$). Therefore, $F_1$ is not strongly isotopic to $F_2$, which means that it is possible to construct inequivalent planar functions from known ones by the isotopism of corresponding presemifield.

%

\begin{remark}
    For Dickson \cite{Dickson1906}, Albert \cite{Albert1960}, Ganley \cite{Ganley1981} and Cohen-Ganley \cite{Cohen-Ganley1982} commutative semifields, we did not find such $\lambda$ to construct affine-inequivalent functions $f_1$ and $f_2$ as defined above on $\mathbb{F}_{3^{2m}}$, where $m=2, 3$. For the Budaghyan-Helleseth-Bierbrauer (BHB) semifields \cite{Bierbrauer2009-3,B-H2008,BudaghyanHelleseth2010} of order $3^6$, such $\lambda$ can also be found. For any other larger $m$, it is beyond our computation capacity.
\end{remark}

\begin{remark}
    We find that $f_2$ is equivalent to the planar function from BHB semifield of order $3^6$, which means BHB semifield and LMPTB semifield of order $3^6$ are isotopic but not strongly isotopic.
\end{remark}

\section*{Acknowledgement}
    We are grateful to R. Coulter and R. Knarr for pointing out the result in \cite{pieper-seier_remarks_1999}.





\bibliographystyle{model1b-num-names}
\bibliography{ref-sc}







\end{document}